\documentclass[12pt, reqno]{amsart}
\usepackage{verbatim}
\usepackage{graphicx}
\usepackage{enumerate}

\usepackage{accents}

\newtheorem{theorem}{Theorem}
\newtheorem{lemma}[theorem]{Lemma}

\newtheorem{conjecture}{Conjecture}
\newtheorem{question}[conjecture]{Question}
\newtheorem*{claim*}{Claim}

\theoremstyle{definition}
\newtheorem*{definition*}{Definition}

\newtheorem{remark}[theorem]{Remark}

\newtheorem{example}[theorem]{Example}

\newcommand\ns[1]{ \left\{ {#1} \right\} }

\newcommand{\Z}{{\mathbb Z}}

\newcommand{\N}{{\mathbb N}}

\newcommand{\eqd}{\stackrel{d}{=}}

\newcommand{\cg}{c}

\newcommand\garbage[1]{}

\renewcommand{\P}{{\mathbb P}}

\newcommand{\dff}[1]{\textbf{\emph{#1}}}

\newcommand{\erk}{\hfill \ensuremath{\Diamond}} 

\newcommand{\bb}[1]{\mathbb{#1}}
\newcommand{\frgr}[1]{\mathbb{F}_{#1}}
\newcommand{\frgrp}[1]{\mathbb{F}_{#1}^{\prime}}

\newcommand{\uc}{\mathbf{u}}

\begin{document}

\title[Equivariant thinning]{Equivariant thinning over a  free group}
\author{Terry Soo and Amanda Wilkens}

\address{Department of Mathematics,
University of Kansas,
405 Snow Hall,
\indent 1460 Jayhawk Blvd,
Lawrence, Kansas 66045-7594}

\email{tsoo@ku.edu}
\urladdr{http://www.tsoo.faculty.ku.edu}
\email{awilkens@ku.edu}
\urladdr{http://www.people.ku.edu/~a548w850/}
\thanks{Funded in part by a New Faculty General Research Fund}


\begin{abstract}
We construct  entropy increasing monotone  factors in the context of a Bernoulli shift over the free group of rank at least two.  
\end{abstract}

\maketitle

\section{Introduction}

Let $\kappa$ be a probability measure on a finite set $K$.   We will mainly be concerned with the simple case where $K = \ns{0,1}$, where we call $\kappa(1):= \kappa(\ns{1}) \in (0,1)$ the \dff{intensity} of $\kappa$.   Let $G$ be a group.      A \dff{Bernoulli shift over $G$ with base $(K, \kappa)$} is the measure-preserving system $(G, K^G, \kappa^G)$, where $G$ acts on $K^G$ via $(gx)(f) = x(g^{-1}f)$	 for $x \in K^G$ and $g,f \in G$.  Let  $\iota$ be a probability measure of lower intensity.  We say that a measurable map $\phi: K^G \to K^G$ is an  \dff{equivariant thinning from $\kappa$ to $\iota$} if $\phi(x)(g) \leq x(g)$ for all $x \in K^G$ and $g \in G$,   the push-forward of $\kappa^G$ under $\phi$ is $\iota^G$, and $\phi$ is equivariant $\kappa^G$-almost-surely; that is, on a set of full-measure, $\phi \circ g = g \circ \phi$ for all $g \in G$.  

\begin{theorem} 
	\label{group}
	 Let $\kappa$ and $\iota$ be probability measures on $\ns{0,1}$ and  $\iota$ be of lower intensity.  For Bernoulli shifts over the free group of rank at least two, there exists an equivariant thinning from $\kappa$ to $\iota$.  
\end{theorem}

Theorem \ref{group} does not hold with such generality  in the case of a Bernoulli shift over an amenable group like the integers.     Recall that the \dff{entropy} of a probability measure $\kappa$ on a finite set $K$ is given by
$$ H(\kappa) :=  -\sum_{i \in K} \kappa(i) \log \kappa(i).$$

\begin{theorem}[Ball \cite{Ball}, Soo \cite{Soo2016}]
	\label{KS}
	  Let $\kappa$ and $\iota$ be probability measures on $\ns{0,1}$ and  $\iota$ be of lower intensity.  For Bernoulli shifts over the integers, there exists an equivariant thinning from $\kappa$ to $\iota$ if and only if $H(\kappa) \geq H(\iota)$.   	
\end{theorem}
 
 In Theorem \ref{KS}, the necessity of $H(\kappa) \geq H(\iota)$ follows easily from the classical theory of Kolmogorov-Sinai entropy  \cite{MR2342699,MR0304616}, which we now recall.    Let $G$ be a group and let  $\kappa$ and $\iota$ be probability measures on a finite set $K$.      An equivariant map $\phi$ is a \dff{factor} from $\kappa$ to $\iota$ if the push-forward of $\kappa^G$ under $\phi$ is $\iota^G$, and 
 is an \dff{isomorphism} if $\phi$ is a bijection and its inverse also serves as a  factor from $\iota$ to $\kappa$.     In the case $G = \Z$, Kolmogorov proved that entropy is non-increasing under factor maps; this implies the necessity of $H(\kappa) \geq H(\iota)$ in Theorem \ref{KS}.  
 Furthermore,    Sinai \cite{MR2766434} proved that there is  a factor from $\kappa$ to $\iota$ if  $H(\kappa) \geq H(\iota)$, and Ornstein \cite{Ornstein}  proved there is an isomorphism from $\kappa$ to $\iota$ if and only if $H(\kappa) = H(\iota)$.    Thus entropy is a complete invariant for Bernoulli shifts over $\Z$.     Ornstein and Weiss \cite{MR910005} generalized these results to the case where $G$ is an amenable group.    See also Keane and 
 Smorodinsky 
 for concrete constructions of factor maps and isomorphisms \cite{keanea, keaneb}.   
 
The sufficiency of $H(\kappa) > H(\iota)$ in Theorem \ref{KS} was first proved by Ball \cite{Ball}.
  The existence of an isomorphism that is also an 
  equivariant
  thinning in the equal entropy case was proved by Soo \cite{Soo2016}.     Let us remark that the factor maps given in standard proofs of the   Sinai and Ornstein theorems will not in general be monotone; that is, they may not satisfy $\phi(x)(i )\leq x(i)$ for all $x \in \ns{0,1}^{\Z}$ and $i \in \Z$.     
 
 Towards the end of their 1987 paper, Ornstein and Weiss \cite{MR910005} give a simple
  but remarkable example of an entropy increasing factor in the case where $G$ is the free group of rank at least two,  which is further elaborated upon by  Ball \cite{ball3}.   It was an open question until recently whether all Bernoulli shifts over a free group of rank at least two are isomorphic.   This question was answered negatively by Lewis Bowen \cite{Lbowen} in 2010, who proved that although entropy can increase under factor maps, in the context of a free group with rank at least two, it is still a complete isomorphism invariant.     Recently, there has been much interest in studying factors in the non-amenable setting; see Russell  Lyons \cite{iidtrees} for more information.

  Our proof of Theorem \ref{group}  will make use of a variation of the  Ornstein and Weiss example in Ball \cite{ball3} and a primitive version of a  marker-filler type  construction, in the sense of Keane and Smorodinsky \cite{keanea, keaneb}.    Our construction uses randomness already present in the process in a  careful way as to mimic a construction that one would make if additional
  independent randomization 
  were  
  available.  This  approach was 
  taken by Holroyd, Lyons, and Soo \cite{MR2884878}, Angel, Holroyd, and Soo \cite{MR2736350}, and Ball \cite{MR2133893} for defining equivariant thinning in  the context of Poisson  point processes.

\section{Tools}

\subsection{Coupling}

Let $(A, \alpha)$ and $(B, \beta)$ be probability spaces.  A \dff{coupling} of $\alpha$ and $\beta$ is a probability measure on the product space $A \times B$  which has $\alpha$ and $\beta$ as its marginals.   For a random variable $X$, we will refer to the measure $\P(X \in \cdot)$ as the \dff{law} or the  \dff{distribution} of $X$. If two random variables $X$ and $Y$ have the same law, we write $X \eqd Y$.     Similarly, a \dff{coupling} of random variables $X$ and $Y$ is a pair of random variables $(X', Y')$, where $X'$ and $Y'$ are  defined on the same probability space and have the same law as $X$ and $Y$, respectively.  Thus a  coupling of random variables gives a coupling of the laws of the random variables.  Often we will refer to the law of a pair of random variables as the \dff{joint distribution} of the random variables.      In the case that $A=B$ and $A$ is a partially ordered by the relation $\preceq$, we say that a coupling $\gamma$ is \dff{monotone} if $\gamma\ns{(a,b) \subset A \times A :  b \preceq a )} = 1$. 
We will always endow  the space of binary sequences $\ns{0,1}^I$ indexed by a set $I$ with the partial order $x \preceq y$ if and only if $x_i \leq y_i$ for $i \in I$.

\begin{example}[Independent thinning]
	\label{indepthin}
	Let $\kappa$ and $\iota$ be probability measures on $\ns{0,1}$, where $\kappa(1) := p \geq \iota(1) :=q$.  Let $r := \tfrac{p-q}{p}.$     Then the measure $\rho$ on $\ns{0,1} ^2$ given by  
	$$\rho(0,0) = 1-p, \  \rho(0,1)=0,  \   \rho(1,0) =  rp, \text{ and }  \rho(1,1) = (1-r)p$$ 
	is a monotone coupling of $\kappa$ and $\iota$.  Thus under $\rho$, a  $1$ is thinned to a $0$ with probability $r$ and kept with probability $1-r$.     
	Clearly, the product measure $\rho^n$ is a  monotone coupling of $\kappa^n$ and $\iota^n$.  We will refer to the coupling $\rho^n$ as the \dff{independent thinning of $\kappa^n$ to $\iota^n$}.    \erk

\end{example}

The following simple lemma is one of the main ingredients  in the proof of  Theorem \ref{group}. In it we construct a coupling of  $\kappa^n$ and $\iota^n$ for $n$ sufficiently large which will allow us to extract spare randomness from a related coupling of $\kappa^G$ and $\iota^G$.
We will write $0^n1^m$ to 
indicate 
the binary sequence of length $n+m$ of $n$ zeros followed by $m$ ones.

\begin{lemma}[Key coupling]
\label{fund}	
	   Let $\kappa$ and $\iota$ be probability measures on $\ns{0,1}$, where $\kappa$ is of greater intensity.      For $n$ sufficiently large, there exists a monotone coupling $\gamma$ of $\kappa^n$ and $\iota^n$ such that
	$$\gamma(100^{n-2}, 0^n) = \kappa^n(100^{n-2})$$
	and  $$\gamma(010^{n-2}, 0^n) = \kappa^n(010^{n-2}).$$
\end{lemma}

\begin{proof}
	Let $p=\kappa(1)$, $q= \iota(1)$, and $\rho^n$ be the independent thinning of $\kappa^n$ to $\iota^n$ as in Example \ref{indepthin}.   We will perturb  $\rho^n$ to give the required coupling.  We specify a probability  measure $\varrho$ on $\ns{0,1} ^n \times \ns{0,1}^n$ by stating that it agrees with $\rho^n$ except on the points $(100^{n-2}, 0^n)$,   $(010^{n-2}, 0^n)$,  $(100^{n-2}, 100^{n-2})$, and $(010^{n-2}, 010^{n-2})$, where we specify that
$$\varrho(100^{n-2}, 0^n  )  = \varrho(010^{n-2}, 0^n )  = p(1-p)^{n-1}      $$
and
$$ \varrho(100^{n-2}, 100^{n-2}) = \varrho(010^{n-2}, 010^{n-2}) = 0.$$
Thus $\varrho$ is almost a monotone coupling of $\kappa^n$ and $\iota^n$, except that from our changes to $\rho^n$ we have
\begin{equation*}\begin{split}  
	\sum_{x \in \ns{0,1} ^n} \varrho(x, 0^n) & = \sum_{x \in \ns{0,1} ^n} \rho^n(x, 0^n) -\rho^n(100^{n-2},0^n)-\rho^n(010^{n-2},0^n)\\
	& \qquad\qquad +\varrho(100^{n-2},0^n)+\varrho(010^{n-2},0^n)\\
	& = (1-q)^n + 2p(1-p)^{n-1}(1-r),
\end{split}\end{equation*}
and 
\begin{equation*}\begin{split}  
\sum_{x \in \ns{0,1} ^n} \varrho(x, 100^{n-2}) & = \sum_{x \in \ns{0,1} ^n} \rho^n(x, 100^{n-2}) -\rho^n(100^{n-2},100^{n-2})\\
&\qquad\qquad +\varrho(100^{n-2},100^{n-2})\\
& = q(1-q)^{n-1}  -  p(1-p)^{n-1}(1-r)+0\\
&= \sum_{x \in \ns{0,1} ^n} \varrho(x, 010^{n-2}),
\end{split}\end{equation*}
where $r = \tfrac{p-q}{p}$.

We perturb $\varrho$ to obtain the desired coupling $\gamma$.  Consider the set $B_1$ of all binary sequences of length $n$, where $x \in B_1$ if and only if $x_1=1$, $x_2=0$, and  $\sum_{i=3}^n x_i = 1$.    
Similarly, let $B_2$ be the set of all binary sequences of length $n$, where $x \in B_2$ if and only if $x_1=0$, $x_2=1$, and  $\sum_{i=3}^n x_i = 1$.  The sets $B_1$ and $B_2$ are disjoint, and each have cardinality $n-2$.   

 For $x \in B_1 \cup B_2$, 
$$\varrho(x, 0^n) = \rho^n(x, 0^n) =  p^2(1-p)^{n-2}r^2,$$
for $x \in B_1$,
$$\varrho(x, 100^{n-2}) = \rho^n(x, 100^{n-2})= p^2(1-p)^{n-2}r(1-r),$$
and
for $x \in B_2$,
$$\varrho(x, 010^{n-2}) = \rho^n(x, 010^{n-2})= p^2(1-p)^{n-2}r(1-r).$$
Note that for $n$ sufficiently large 
$$\sum_{x \in B_1 \cup B_2}\varrho(x, 0^n) =  2(n-2)  p^2(1-p)^{n-2}r^2   >2p(1-p)^{n-1}(1-r).$$

Let $\gamma$ be equal to $\varrho$ except on the set of points 
$$\ns{(x,0^n) : x \in B_1 \cup B_2}  \cup \ns{(x, 100^{n-2}):  x \in B_1}  \cup \ns{(x, 010^{n-2}): x \in B_2  },$$ where we make the following adjustments.  
For $x \in B_1 \cup B_2$, set  
$$\gamma(x, 0^n) =  p^2(1-p)^{n-2}r^2 - \frac{p(1-p)^{n-1}(1-r)}{n-2} >0,$$
 for $x \in B_1$, set 
$$\gamma(x, 100^{n-2}) =  p^2(1-p)^{n-2}(1-r)r + \frac{p(1-p)^{n-1}(1-r)}{n-2},$$
and for $x \in B_2$, set 
$$\gamma(x, 010^{n-2}) =  p^2(1-p)^{n-2}(1-r)r + \frac{p(1-p)^{n-1}(1-r)}{n-2}.$$
That $\gamma$ has the required properties follows from its construction.
\end{proof}

To illustrate the utility of Lemma \ref{fund}, we will give a different proof of the following result of Peled and Gurel-Gurevich \cite{GGP}.    Let $\N = \ns{0, 1, 2, \ldots}$.  

\begin{theorem}[Peled and Gurel-Gurevich \cite{GGP}]
	\label{thick}
	Let $\kappa$ and $\iota$ be probability measures on $\ns{0,1}$, where $\kappa$ is of greater intensity.   There exists a measurable map $\phi: \ns{0,1}^{\N} \to \ns{0,1}^{\N}$ such that the push-forward of $\kappa^{\N}$ under $\phi$ is $\iota^{\N}$ and  $\phi(x)(i) \leq x(i)$ for all $x \in \ns{0,1}^{\N}$ and all $i \in \N$.  \end{theorem}

We note that in \cite[Theorem 1.3]{GGP}, they use the dual terminology of {\em thickenings}; their equivalent theorem states that for probability measures   $\iota$ and $\kappa$  on $\ns{0,1}$, where $\iota$ is of lesser intensity, there is a measurable map $\phi: \ns{0,1}^{\N} \to \ns{0,1}^{\N}$ such that the push-forward of $\iota^{\N}$ under $\phi$ is $\kappa^{\N}$ and  $\phi(x)(i) \geq  x(i)$ for all $x \in \ns{0,1}^{\N}$ and all $i \in \N$.

In the proof of Theorem \ref{thick}, we will make use of the following two  lemmas.   
We say that a random variable $U$ is \dff{uniformly distributed} in $[0,1]$ 
if the probability that $U$ lies in a Borel subset of the unit interval is given by the Lebesgue measure of the set.
   
\begin{lemma}
	\label{function}
Let $(X, Y)$ be a pair of  discrete random variables taking values on the finite set $A \times B$ with joint distribution $\gamma$.   There exists a measurable function $\Gamma:A \times [0,1] \to B$ such that if $U$ is uniformly distributed 
in 
$[0,1]$ and independent of $X$,  then $(X, \Gamma(X, U))$ has joint distribution $\gamma$.  
\end{lemma}

\begin{proof}   Assume that $\P(X = a) >0$, for all $a \in A$.   Let $B = \ns{b_1 , \ldots, b_n}$.        For each $a \in A$, let 
	$$q_a(j):= \P(Y \in \ns{b_1, \ldots, b_j} | X=a) = \frac{\P  (Y \in \ns{b_1, \ldots, b_j} , X=a)}{  \P( X=a)  }$$ for all $1 \leq j \leq n$.   Set $q_a(0) =0$ and note that $q_a(n) =1$,  so that 
	$$\P(q_a(j-1) \leq U < q_a(j)    ) = \frac{\P(Y=b_j, X=a)}{\P(X=a)}.$$       For each $ 1\leq j \leq n$, let   
	\begin{equation*}
	\Gamma(a, u) := b_j \text{ if } q_a(j-1) \leq u < q_a(j)  . 
\qedhere
\end{equation*}
%
\end{proof}

 We call a  $\ns{0,1}$-valued random variable a \dff{Bernoulli random variable}.   The following lemma allows us to code sequences of independent coin-flips into sequences of uniformly distributed random variables.  

\begin{lemma}
	\label{rep}
	There exists a measurable  function $\cg: \ns{0,1}^{\N} \to [0,1]^{\N}$ such that if $B =(B_{i})_{i \in \N}$ is a sequence of i.i.d.\ Bernoulli random variables with mean $\tfrac{1}{2}$, then $(\cg(B)_i)_{i \in \N}$ is a sequence of  i.i.d.\ random variables that are uniformly distributed in $[0,1]$.	
\end{lemma}

\begin{proof}
	The result follows from the Borel isomorphism theorem.  See  \cite[Theorem 3.4.23]{borel} for more details.  
\end{proof}

\begin{proof}[Proof of Theorem \ref{thick}]  Let $\gamma$ be the monotone coupling of $\kappa^n$ and $\iota^n$ given by Lemma \ref{fund}, so that $\gamma$ is a measure on $\ns{0,1} ^n \times \ns{0,1}^n    \equiv (\ns{0,1} \times \ns{0,1})^n$.   Thus the product measure $\gamma^2$ is a monotone coupling of $\kappa^{2n}$ and $\iota^{2n}$ and $\gamma^{\N}$ gives a monotone coupling of $\kappa^{\N}$ and $\iota^{\N}$.    We will modify the coupling $\gamma^{\N}$ to become the required map $\phi$.  In order to do this, it will be
 easier to think in terms of random variables rather than measures.    
	
Let $X =(X_i)_{i \in \N}$ be an	i.i.d.\ sequence of Bernoulli random variables with mean $\kappa(1)$.  For each $j \geq 0$,    let $$X^{j} := (X_{jn}, \ldots, X_{(j+1)n -1}  ),$$ so that the random variables are partitioned into blocks of size $n$.  Let $U = (U_i)_{i \in \N}$ be an  i.i.d.\ sequence of random variables that are uniformly distributed in $[0,1]$.  Also assume that $U$ is independent of  $X$, and let $Y =(Y_i)_{i \in \N}$ be an	i.i.d.\ sequence of Bernoulli random variables with mean $\iota(1)$.

By Lemmas \ref{fund} and \ref{function},   let $\Gamma :\ns{0,1} ^n \times [0,1] \to \ns{0,1}^n$ be a measurable  map such that $(X^{1}, \Gamma(X^{1}, U_1))$ has joint law $\gamma$ and  $\Gamma(w, v)=0^n$ for all $v \in [0,1]$ if $w \in \ns{100^{n-2}, 010^{n-2}}$.     We have that 
   $$\big(X,  (\Gamma(X^i, U_i))_{i \in \N}  \big)$$ gives a monotone coupling of $X$ and $Y$ with law $\gamma^{\N}$.   
  
  For each $j \in \N$, call  $X^j$ \dff{special} if $X^{j} \in   \ns{100^{n-2}, 010^{n-2}}$ and let $S \subset \N$ be the random set of $j \in \N$ for which $X^j$ are special.  Note that almost surely, $S$ is an infinite set.    Let $\bar{X} = (\bar{X}_i)_{i \in \N}$ be the sequence of binary digits such that $\bar{X}^j = X^j$ if $j \not \in S$ and $\bar{X}^j = 0^n$ if $j \in S$.  We have that 
  $$ (\Gamma(X^i, U_i))_{i \in \N} =  (\Gamma(\bar{X}^i, U_i))_{i \in \N}.$$

    Let $(s_i)_{i \in \N}$ be the enumeration of $S$, where $s_0 < s_1 < s_2 < s_3  \cdots$.   Consider the sequence of random variables given by
  $$b(X) := (\mathbf{1}[X^{s_i}   =100^{n-2}   ])_{i \in \N} = (X_{s_i n})_{i \in \N}$$
  Since $100^{n-2}$ and $010^{n-2}$ occur with equal probability, we have that $b(X)$ is an i.i.d.\ sequence of Bernoulli random variables with mean $\tfrac{1}{2}$.  Furthermore, we have that $b(X)$ is independent of $\bar{X}$, since $b(X)$ only depends on the values of $X$ on the special blocks.   Let $\cg$ be the function from Lemma \ref{rep}, so that $\cg(b(X)) \eqd U$.    Since $b(X)$ is independent of $\bar{X}$, 
  %
  \begin{eqnarray*}
  [\Gamma(X^i, U_i)]_{i \in \N} &=&  [\Gamma(\bar{X}^i, U_i)]_{i \in \N}  \\
   &\eqd& \big[\Gamma(\bar{X}^i, \cg(b(X))_i)\big]_{i \in \N} \\
   &=& \big[\Gamma({X}^i, \cg(b(X))_i)\big]_{i \in \N}.
  \end{eqnarray*}
  Thus $\Big(X, \big[\Gamma({X}^i, \cg(b(X))_i)\big]_{i \in \N}\Big)$ is another  monotone coupling of $X$ and $Y$.  
  Hence, we define 
  $$\phi(x)  := \big[\Gamma({x^i}, \cg(b(x))_i  )\big]_{i \in \N }$$
  for all $x \in \ns{0,1}^{\N}$
  when  
  the set $S$ is infinite, and set $\phi(x)=0^{\N}$ when $S$ is finite--an event that occurs with probability zero. 
   \end{proof}

\subsection{Joinings}

Let $T: \ns{0,1}^{\Z} \to \ns{0,1}^{\Z}$ be the left-shift given by $(Tx)_i = x_{i+1}$ for all $x \in \ns{0,1}^{\Z}$ and all  $i \in \Z$.   Let $\kappa$ and $\iota$ be probability measures on $\ns{0,1}$.   A \dff{joining} of $\kappa^{\Z}$ and $\iota^{\Z}$ is a coupling $\varrho$ of the two measures with the additional property that  $\varrho \circ (T\times T) = \varrho$.    We will  make use of the following joining in the proof of Theorem \ref{group}.

\begin{example} \label{partition}
	Let $\kappa$ and $\iota$ be probability measures on $\ns{0,1}$.  Assume that the intensity of $\kappa$ is greater than the intensity of $\iota$.  
Let $x \in \ns{0,1}^{\Z} $, and let $n$ be sufficiently large as  in Lemma \ref{fund}.  Call the subset $[j, j+2n+1] \subset \Z$ a \dff{marker} if $x_i  =0$ for all $i \in [j, j+2n]$ and $x_{j+2n+1}=1$.  Notice that two distinct markers have an empty intersection.  Call an interval a \dff{filler} if it is nonempty and lies between two markers.  Thus each $x \in \ns{0,1}^{\Z}$ partitions $\Z$ into intervals of markers and fillers.   Call a filler \dff{fitted} if it is of size $n$, and call  a filler  \dff{special} if it is both fitted and of the form $100^{n-2}$ or $010^{n-2}$.

Let $X$ have law $\kappa^{\Z}$ and $Y$ have law $\iota^{\Z}$.  In what follows we describe explicitly how to obtain a monotone joining of $X$ and $Y$, where the independent   thinning is used everywhere, except at the fitted fillers, where the coupling from Lemma \ref{fund} is used.   
 Let $U=(U_i)_{i \in \Z}$ be an i.i.d.\ sequence of random variables that are uniformly distributed in $[0,1]$ and independent of $X$.  By Example \ref{indepthin}  and Lemma \ref{function}, let $R: \ns{0,1} \times [0,1] \to \ns{0,1}$ be a measurable function such that $R(X_1, U_1) \leq X_1$ is a Bernoulli random variable with mean $\iota(1)$.    Let $\Gamma$ and $\gamma$ be as in the proof of Theorem \ref{thick}, so that 
$$\big( (X_1, \ldots, X_n), \Gamma( X_1, \ldots, X_n, U_1) \big)$$ has law $\gamma$.  
  Consider the function $\Phi: \ns{0,1} ^{\Z} \times [0,1] ^{\Z}  \to    \ns{0,1} ^{\Z}$ defined by
$\Phi(x, u)_i = R(x_i, u_i)$ if $i$ is not in a fitted filler.  For $(j, j+1, \ldots, j+n)$ in a fitted filler, we set 
$$ (\Phi(x,u)_j, \ldots, \Phi(x,u)_{j+n}) = \Gamma(x_j, \ldots, x_{j+n}, u_j   ).$$
The law of $X$ restricted to a filler interval is the law of a finite sequence of i.i.d.\ 
Bernoulli random variables with mean $\kappa(1)$, conditioned not to contain a marker.    Note that since a fitted interval  is of size $n$, and a marker is of size $2n+1$,  the law of $X$ restricted to a fitted interval is just the law of a finite sequence of i.i.d.\ 
Bernoulli random variables with mean $\kappa(1)$.  Furthermore, conditioned on the locations of the markers, the restrictions of $X$ to each filler interval are independent (see for example Keane and Smorodinsky \cite[Lemma 4]{keanea} for a detailed proof).   Hence, $\Phi(X, U) \eqd Y$.  In addition, since  all the couplings involved are monotone, we easily have that $\Phi(X,U)_i \leq X_i$ for all $i \in \Z$. \erk  
\end{example}

\begin{remark}
	\label{KS2}
	To emphasize the strong form of independence in Example \ref{partition}, we note that 
	if $A=(A_i)_{i\in\bb Z}$ are independent Bernoulli random variables with mean $\tfrac{1}{2}$ that are independent of $X$, then $(A_{jn})_{j\in S}$ has the same law as $(X_{jn})_{j\in S}$. Recall if $j\in S$ then $X^{j} = (X_{jn}, \ldots, X_{(j+1)n -1}  )$ is special.
	In addition, if $X'$ is such that $X'_i= X_i$ for every $i$ not in a special filler of $X$ and on each special filler of $X$ we set $X'_{jn} = A_{jn}$, $X'_{jn+1} =1-A_{jn}$, and
	\begin{equation*}
		X'_{jn+2}=X'_{jn+3} = \cdots =X'_{(j+1)n-1}= 0,
	\end{equation*}
	then $X' \eqd X$.
	Thus we can independently resample on the special fillers without affecting the distribution of $X$.  \erk
\end{remark}

\subsection{The example of Ornstein and Weiss}

Let $\mathbb{F}_r$ be the free group of rank $r \geq 2$.  Let $a$ and $b$ be two of its generators.   The Ornstein and Weiss \cite{MR910005}  entropy increasing factor map is  given by   
$$\phi(x)(g) = (x(g) \oplus  x(ga), x(g) \oplus x(gb))$$
for all  $x \in \ns{0,1}^{\mathbb{F}_r}$ and all $g \in \mathbb{F}_2$, where 
$$\phi: \ns{0,1}^{\mathbb{F}_r} \to ( \ns{0,1} \times \ns{0,1} )^{\mathbb{F}_r} \equiv \ns{00,01,10,11}^{\mathbb{F}_r}$$
 pushes the uniform product measure $(\tfrac{1}{2}, \tfrac{1}{2}) ^{\mathbb{F}_r}$    forward to the uniform product measure $(\tfrac{1}{4}, \tfrac{1}{4}, \tfrac{1}{4}, \tfrac{1}{4}) ^{\mathbb{F}_r}$; 
the required independence follows from the observation that if $m \oplus n := m + n \bmod 2$, if $X$, $X'$, and $Y$ are
independent Bernoulli random variables with mean $\tfrac{1}{2}$, and if $Z := X \oplus Y$  and  $Z' := X' \oplus Y$, then $Z$ and $Z'$ are independent, even though they both depend on $Y$. 

Ornstein and Weiss's example can be iterated to produce an infinite number of bits at each vertex in the following way.    As in Ball \cite[Proposition 2.1]{ball3}, we will define  $\phi_k :  \ns{0,1} ^{\mathbb{F}_r} \to  (   \ns{0,1} ^{k}  ) ^{\mathbb{F}_r}$  inductively for $k \geq 2$.    Let  $\tilde{\phi_k}:   \ns{0,1} ^{\mathbb{F}_r} \to    \ns{0,1}  ^{\mathbb{F}_r}$ be the last coordinate of $\phi_k$ so that $\tilde{\phi_k}(x)(g) = [\phi_k(x)(g)]_k$ for all $x \in \ns{0,1}^{\mathbb{F}_r}$ and all $g \in \mathbb{F}_2$.
   Set $\phi_2 = \phi$ .  
 For $k \geq 3$, let $\phi_k$ be given by 
$$ \phi_k(x)(g) = \Big( [{\phi_{k-1} (x)(g)}]_1, \ldots, [{\phi_{k-1}(x)(g) }]_{k-2}, (\phi \circ \tilde{ \phi}_{k-1})(x)(g) \Big)$$
for all  $x \in \ns{0,1}^{\mathbb{F}_r}$ and all $g \in \mathbb{F}_2$.
At each step we are saving one bit to generate two new bits using the original map $\phi$.  
%
   The map  $\phi_k$ pushes   the uniform product measure $(\tfrac{1}{2}, \tfrac{1}{2}) ^{\mathbb{F}_r}$ forward to the uniform product measure on   $(   \ns{0,1} ^{k}  ) ^{\mathbb{F}_r}$. By taking the limit, we obtain the mapping $$\phi_{\infty}: \ns{0,1} ^{\mathbb{F}_r} \to  (   \ns{0,1} ^{\Z^{+}}  ) ^{\mathbb{F}_r}$$ which yields a sequence of i.i.d.\ fair bits at each coordinate $g\in \mathbb{F}_2$,  independently.    Note that $\phi_{\infty}(x)(g)_k = \phi_{n}(x)(g)_k$ for all $n > k$.  
In our proof of Theorem \ref{group} we will use this iteration, which Ball attributes to Tim\'ar.  

\section{Proof of the main theorem}

\begin{proof}[Proof of Theorem \ref{group}]

Let $r \geq 2$.  
We begin by extending the same monotone joining defined in Example \ref{partition} to a monotone joining of $\kappa^{\frgr r}$ and $\iota^{\frgr r}$. Let $X$ have law $\kappa^{\frgr r}$ and $Y$ have law $\iota^{\frgr r}$;   then $X = (X_g)_{g \in \frgr r} = (X(g))_{g \in \frgr r}$ are i.i.d.\ Bernoulli random variables with mean $\kappa(1)$.    
 As in the Ornstein and Weiss example, it will be sufficient to use only two generators $a$ and $b$ in the expression of our equivariant thinning. 
We refer to the string of generators and their inverses that make up the representation of an element in $\frgr r$ as a \dff{word}, and the individual generators and inverses as \dff{letters}. We call a word \dff{reduced} if its string of letters has no possible cancellations.

Consider $\frgr r$ as being partitioned into infinitely many $\bb Z$ copies $Z(w)$ in the following way. Let $\frgrp r $ be the set of  reduced words in $\frgr r$ that do not end in either $b$ or $b^{-1}$. For each  $w \in \frgrp r$, 
 set $Z(w):=\{wb^i\}_{i\in\Z}$. 
Indeed, any element in $\frgr r$ may be written as $wb^i$ for  unique reduced $w \in \frgrp r  $ and $i \in \Z$.   

 Let $n$ be sufficiently large for the purposes of Lemma \ref{fund}.    We define markers, fillers, fitted fillers, and special fillers on each of the $\Z$ copies in the obvious way.  For example, if $x \in \ns{0,1}^{\frgr r}$ and  $w \in \frgrp r$, then the set $\ns{wb^j, \ldots, wb^{j+2n+1}} $ is a marker  if $ x(wb^i)=0$ for all $i \in [j, 2n]$ and $x(wb^{2n+1}) =1$.  

 Let $U'=(U'_g)_{g \in{\frgr r}}$ be  i.i.d. uniform random variables independent of $X$.   Let $\Phi$ be as in Example \ref{partition}.  
Define $\hat\Phi : \{0,1\}^{\frgr r}\times [0,1]^{\frgr r}\rightarrow \{0,1\}^{\frgr r}$ by
$$\hat\Phi (x,u')_{wb^i}=\Phi \big(x(Z(w)),u'(Z(w) )\big)_i$$
for all $w \in \frgrp r$ and all $i \in \Z$, where $x(Z(w)):= (x(wb^j))_{j \in \Z}$ and $u'(Z(w)) := (u'(wb^j))_{j \in \Z}$.   
Thus we have the monotone joining $\Phi$ on each $\bb Z$ copy $Z(w)$ in $\frgr r$, so that 
 \begin{equation}
 \label{eqd}
 \hat\Phi (X,U')\eqd Y
 \end{equation}
 and $\hat\Phi (X,U')_g  \leq X_g$ for all $ g \in\frgr r$. Additionally,   since $\Phi$ is a joining, the joint law of 
 $(X, \hat\Phi (X,U'))$ is invariant under $\frgr r$-actions.

Recall that  a special filler has length exactly $n$, and the filler has two choices of values  $010^{n-2}$ or $100^{n-2}$, which occur with equal probability. We define an  \dff{initial vertex} of a special filler in $Z(w)$ to be an element $wb^{n_0} \in Z(w)$  where the entire special filler takes values sequentially at vertices on the minimal  path from $wb^{n_0}$ to $wb^{n_0+n}$.      For each $x \in \ns{0,1}^{\frgr r}$,    let $V=V(x)$ be the set of initial vertices in $\frgr r$. 
Note that as in Example \ref{partition}, the law of $X$ restricted to a fitted interval is just the law of a finite sequence of i.i.d.\ Bernoulli random variables with mean $\kappa(1)$.  Furthermore, conditioned on the locations of the markers, the restrictions of $X$ to each filler interval are independent. 
Thus for all $v \in V(X)$, $X(v)$ is a Bernoulli random variable with mean $\frac{1}{2}$, and conditioned on $V(X)$, the  random variables $(X(v))_{v \in V}$ are independent. 

We have the same strong form of independence here as emphasized in Remark \ref{KS2} for Example \ref{partition}, again by Keane and Smorodinsky \cite[Lemma 4]{keanea}. This is key in our construction:   we will use the Bernoulli random variables $(X(v))_{v \in V}$ to build deterministic substitutes for $U'$.  

Now we adapt the iteration of the Ornstein and Weiss example to  assign    a sequence of i.i.d.\ Bernoulli random variables to each $v \in V$.   For each $v \in V$, let $k$ be the smallest  positive integer such that  $va^k \in V$; set $\alpha(v) = va^k$.    Similarly, let $k'$ be the smallest positive integer such that $vb^{k'}\in V$ and set $\beta(v) = v b^{k'}.$    
For each $v \in V$, define
$$\psi(x)(v) =  \big( x(v) \oplus x( \alpha(v)), x(v) \oplus x( \beta(v)) \big).$$  Conditioned on $V$, we have that $ (\psi(X))_{v \in V}$ is a family of independent random variables uniformly distributed on $\ns{ 00, 01, 10, 11}$.  We  iterate the map $\psi$ as we did with the Ornstein and Weiss map $\phi$.   Set $\psi_2 = \psi$.    For $k \geq 3$, let 
$$ \psi_{k}(x)(v) =   \Big( [{\psi_{k-1} (x)(v)}]_1, \ldots, [{\psi_{k-1}(x)(v) }]_{k-2}, (\psi \circ \tilde{\psi}_{k-1})(x) (v) \Big),$$ 
where $\tilde{\psi}_{k-1}(x)(v) = [\psi_{k-1}(x)(v)]_{k-1}$ is the last coordinate of $\psi_k$.  
Let $\psi_{\infty}$ be the limit, and let $B_v = \psi_{\infty}(X)(v)$, so that conditioned on $V$,  the random variables $(B_v)_{v \in V}$ are independent, and  each $B_v$ is an i.i.d.\ sequence of  Bernoulli random variables  with mean $\frac{1}{2}$.

For all $x\in\{0,1\}^{\frgr r}$, let $\bar{x}(g) = x(g)$ for all $g$ not in a special filler, and let $\bar{x}(g) = 0$ if $g$ belongs to a special filler.  
It follows from Remark \ref{KS2} that if $B' = (B'_g)_{g \in \frgr r}$ are independent Bernoulli random variables with mean $\tfrac{1}{2}$ independent of $X$, then $(B'_v)_{v \in V(X) }$ has the same law as $(B_v)_{v \in V(X) }$. Moreover,
\begin{equation}
\label{use}
\big(\bar{X}, (B_v)_{v \in V(X) } \big) \eqd \big(\bar{X}, (B'_v)_{v \in V(X) }\big).
\end{equation}

 We assign, in an equivariant way,  one uniform random variable to each element in $\frgr r$ using the randomness provided by  $(B_v)_{v \in V }$.   Let $\cg:\ns{0,1}^{\N} \to [0,1]^{\N}$ be the function from Lemma \ref{rep},
 and let $g\in\frgr r$. 
Then almost surely there exist $v\in V$ and a minimal $j>0$ such that $gb^j=v$;
 set $U_g=\cg(B_v)_{j} $.  
 Define  $\uc:\ns{0,1}^{\frgr r} \to [0,1]^{\frgr r}$ by setting
  $ \uc(X):=(U_g)_{g\in\frgr r}$.  
 Recall that $U' = (U_g')_{g \in \frgr r}$ are independent random variables uniformly distributed in $[0,1]$ independent of $X$.     From \eqref{use},
\begin{equation}
\label{use2}
\big(\bar{X}, {\uc(X)} \big) \eqd \big(\bar{X}, {U'} \big).
\end{equation}

Let $R: \ns{0,1} \times [0,1] \to \ns{0,1}$ and   $\Gamma:\ns{0,1}^n \times [0,1] \to \ns{0,1}^n$ be the functions  that appear in the definition of $ \Phi$ in Example \ref{partition}.  Recall that $R$ facilitated independent thinning and $\Gamma$  the key monotone coupling of Lemma \ref{fund}.  Also recall   $\Gamma(100^{n-2}, t) = 0 = \Gamma(010^{n-2}, t)$ for all $t \in [0,1]$.

Now define $\phi:\{0,1\}^{\frgr r}\rightarrow \{0,1\}^{\frgr r}$ by
$$\phi(x)(g)=R\big(  x(g),\uc(x)(g)  \big)$$
for $g$ not in a fitted filler; if $\ns{wb^i,\ldots,wb^{i+n-1}}$ is a fitted filler, then set
$$(\phi(x)(wb^i),\ldots,\phi(x)(wb^{i+n-1}))=\Gamma\big(  x(wb^i),\ldots, x(wb^{i+n-1}) ,\uc(x)(wb^i)  \big).$$
Note  $\phi$ is defined so that $\phi(x) = \hat{\Phi}(x, \uc(x))$.  The map $\phi$ is equivariant and satisfies $\phi(x)(g) \leq x(g)$ by construction.    It remains to verify that $\phi(X) \eqd Y$.  

 By the definition of $\Gamma$, we have $\phi(X) = \phi(\bar{X})$; that is, all special fillers are sent to $0^n$.   A similar remark applies to the map $\hat{\Phi}$.    From   \eqref{eqd} and \eqref{use2}, 
 $$\phi(X) = \hat{\Phi}({X}, \uc(X)) =\hat{\Phi}(\bar{X}, \uc(X)) \eqd  \hat{\Phi}(\bar{X}, U') = \hat{\Phi}(X, U')\eqd Y.\qedhere$$
\end{proof}

\section{Generalizations and questions}

\subsection{Stochastic domination}

Let $[N] = \ns{0, 1, \ldots, N-1}$ be endowed with the usual total ordering.    Let $\kappa$ and $\iota$ be probability measures on $[N]$.   We say that $\kappa$ \dff{stochastically dominates} $\iota$ if $\sum_{i=0} ^j   \kappa_i \leq \sum_{i=0} ^j \iota_i$ for all $ j \in [N]$.   An elementary version of Strassen's theorem \cite[Theorem 11] {Strassen} gives that $\kappa$ stochastically dominates $\iota$ if and only if there exists a monotone coupling of $\kappa$ and $\iota$.   Notice that in the case $N=2$, we have that $\kappa$ stochastically dominates $\iota$ if and only  if $\iota$ is not of higher intensity than $\kappa$.   Thus Theorem \ref{group} gives a positive answer to a special case of the following question.

\begin{question}
\label{SDgen}
Let $\kappa$ and $\iota$  be probability measures on $[N]$, where $\kappa$ stochastically dominates $\iota$, and $\kappa$ gives positive measure to at least two elements of $[N]$.    Let $G$ be the free group of rank at least two.  Does there exist a measurable  equivariant map $\phi:  [N] ^G \to [N]^G$ such that the push-forward of $\kappa^G$ is $\iota^G$ and $\phi(x)(g) \leq x(g)$ for all $x \in [N] ^G$ and $g \in G$?
\end{question}  

In Question \ref{SDgen}, we call the map $\phi$ a \dff{monotone factor from $\kappa$ to $\iota$}. A necessary condition for the existence of a monotone factor from 
 $\kappa$ to $\iota$ is that $\kappa$ stochastically dominates $\iota$.    In the case $G= \Z$, Ball \cite{Ball} proved that there exists a monotone factor from $\kappa$ to $\iota$ provided that $\kappa$ stochastically dominates $\iota$, $H(\kappa) > H(\iota)$, and $\iota$ is supported on two symbols; Quas and Soo \cite{Qsc} removed the two symbol condition on $\iota$.

In the non-amenable case, where $G$ is a free group of rank at least two, one can hope that Question \ref{SDgen} can be answered positively, without any entropy restriction.    However, the analogue of Lemma \ref{fund} that was key to the proof of Theorem \ref{group}
 does not apply  in the simple case where $\kappa = (0,\tfrac{1}{2}, \tfrac{1}{2})$ and $\iota = (\tfrac{1}{3}, \tfrac{1}{3}, \tfrac{1}{3})$.  In particular, for all $n \geq 1$,  there is no coupling  $\rho$ of $\kappa^n$ and $\iota^n$ for which there exists $x \in \ns{1,2}^n$ and $y \in \ns{0,1,2}^n$ such that $\rho(x,y) = \kappa^n(x) = (\tfrac{1}{2})^n$, since $\rho(x,y) \leq \iota^n(y) = (\tfrac{1}{3})^n$.

\subsection{Automorphism-equivariant factors}

The Cayley graph of  $\mathbb{F}_n$ is  the regular tree $\mathbb{T}_{2n}$ of degree $2n$. We note that $\mathbb{F}_n$ is a strict subset of the group of graph automorphisms of $\mathbb{T}_{2n}$.     The map that we constructed in Theorem \ref{group} is not equivariant with respect to the full automorphism group of $\mathbb{T}_{2n}$.  In particular,  our definition of a marker is not equivariant with respect to the automorphism which exchanges $a$-edges and $b$-edges in $\bb T_{2n}$. However, Ball generalizes  the Ornstein and Weiss example to the full automorphism group in \cite[Theorem 3.3]{ball3} by proving 
that for any $d \geq 3$,  there exists a measurable mapping  $\phi: \ns{0,1} ^{\mathbb{T}_d }  \to [0,1] ^{ \mathbb{T}_d} $ which  pushes the uniform product measure on two symbols  forward to the product measure of Lebesgue measure on the unit interval, equivariant with respect to the group of automorphisms of $\mathbb{T}_d$. Moreover, she proved the analogous result for any tree with bounded degree, no leaves, and at least three ends.    

\begin{question}  Let $T$ be a tree with bounded degree, no leaves, and at least three ends.  Let $\kappa$ and $\iota$ be probability measures on $\ns{0,1}$ and $\iota$ be of lower intensity.     Does there exists    a  thinning  from $\kappa$ to $\iota$ that is   equivariant with respect to the full automorphism group of $T$?

\end{question}

\section*{Acknowledgements}

We thank the referee for carefully reviewing the paper and providing various helpful comments and suggestions.

\bibliographystyle{abbrv}
\bibliography{embedding}

\end{document}